\documentclass[11pt,a4paper]{article}
\usepackage{amsmath}
\usepackage{array,a4,amssymb}
\usepackage{amsthm}
\newtheorem{aaaaa}{DO NOT USE!}
\newtheorem{Corollary}[aaaaa]{Corollary}

\newtheorem{Definition}[aaaaa]{Definition}
\newtheorem{Example}[aaaaa]{Example}
\newtheorem{Lemma}[aaaaa]{Lemma}
\newtheorem{Proposition}[aaaaa]{Proposition}
\newtheorem{Remark}[aaaaa]{Remark}
\newtheorem{Theorem}[aaaaa]{Theorem}
\usepackage{graphicx,booktabs,url,color}
\usepackage{multirow}
\begin{document}
\begin{center}
{\Large{On the number of zeros of multiplicity $r$}}\\
\ \\
Olav Geil and Casper Thomsen\\
Department of Mathematical Sciences\\
Aalborg University\\
\ \\

Email: olav@math.aau.dk and caspert@math.aau.dk
\end{center}
\ \\
\begin{center}
\begin{minipage}{12cm}
{\small{{\textbf{Abstract:}} \ \\
Let $S$ be a finite subset of a field. For multivariate polynomials
the generalized Schwartz-Zippel bound~\cite{augot}, \cite{dvir}
estimates the number of zeros over $S \times \cdots \times S$ counted
with multiplicity. It does this in terms of the total degree, the
number of variables and $| S |$. In the present work we take into
account what is the leading monomial. This allows us to consider more
general point ensembles and most importantly it allows us to produce much
more detailed information about the number of zeros of multiplicity
$r$ than can be deduced from the generalized Schwartz-Zippel bound. We
present both upper and lower bounds. \\


 }}

\end{minipage}
\end{center}
\ \\

\section{Introduction}\label{secintro}
In this paper we consider multivariate polynomials over an arbitrary
field ${\mathbf{F}}$. Our studies focus on the zeros of  given
prescribed multiplicity, a concept to be defined more formally below. 
The definition of multiplicity that we will use relies on the Hasse
derivative. This derivative 
coincides with the usual analytic
derivative in the case of polynomials over the reals. Before recalling the definition of the
Hasse derivative let us fix some notation. Assume we are given a vector of
variables $\vec{X}=(X_1, \ldots ,X_m)$ and a vector $\vec{k}=(k_1,
\ldots , k_m)\in {\mathbf{N}}_0^m$ then we will write
$\vec{X}^{\vec{k}}=X_1^{k_1} \cdots X_m^{k_m}$. We will always assume
that 
$\vec{X}$ and $\vec{Z}$ are vectors of $m$ variables. 
\begin{Definition}
Given $F(\vec{X})\in {\mathbf{F}}[\vec{X}]$ and
$\vec{k} \in {\mathbf{N}}_0^m$ the $\vec{k}$'th
Hasse derivative of $F$, denoted by $F^{(\vec{k})}(\vec{X})$ is the
coefficient of $\vec{Z}^{\vec{k}}$ in $F(\vec{X}+\vec{Z})$. In other words 
$$F(\vec{X}+\vec{Z})=\sum_{\vec{k}} F^{(\vec{k})}(\vec{X})\vec{Z}^{\vec{k}}.$$
\end{Definition}
The concept of multiplicity for univariate polynomials is generalized
to multivariate polynomials in the following way.
\begin{Definition}\label{defmult}
For $F(\vec{X}) \in {\mathbf{F}}[\vec{X}]\backslash \{ \vec{0} \}$ and
$\vec{a}\in {\mathbf{F}}^m$ we define the multiplicity of $F$ at $\vec{a}$
denoted by ${\mbox{mult}}(F,\vec{a})$ as follows. Let $M$ be an
integer such that for every $\vec{k}=(k_1, \ldots ,
k_m) \in {\mathbf{N}}_0^m$ with $k_1+\cdots +k_m < M$, $F^{(\vec{k})}(\vec{a})=0$
  holds, but for some  $\vec{k}=(k_1, \ldots ,
k_m) \in {\mathbf{N}}_0^m$ with $k_1+\cdots +k_m = M$,
$F^{(\vec{k})}(\vec{a})\neq 0$ holds, then 
${\mbox{mult}}(F,\vec{a})=M$. If $F=0$ then we define ${\mbox{mult}}(F,\vec{a})=\infty$.
\end{Definition}
It is of evident interest to investigate for multivariate  polynomials $F$
and a finite ensemble
of points the following questions:
\begin{itemize}
\item[Q1] How many zeros can $F$ have in total when counted with multiplicity?
\item[Q2] How many zeros of a given prescribed multiplicity can $F$ have?
\end{itemize}
Clearly, assuming finite ensembles of points is not a restriction
when $\mathbf{F}$ is a finite field
${\mathbf{F}}_q$. We note that the above questions have important implications in a number of applications, see~\cite{dvir}
and~\cite{pw_expanded}. What we would like to have for certain natural
ensembles of points is bounds on the
number of points in terms of the total degree of $F$ or even better in terms
of ${\mbox{lm}}(F)$. Here, ${\mbox{lm}}(F)$ denotes the leading
monomial of $F$ with respect to some fixed monomial
ordering.\\
The related problem of bounding the number of zeros (counted without multiplicity) has been
dealt with using two completely different approaches. On the one hand a tight bound in terms of the leading
monomial has be derived using the footprint bound from Gr\"{o}bner
basis theory (see~\cite{clo} and \cite{gh}). On the other hand a tight bound in term of the total degree, known
as the Schwartz-Zippel bound, was derived using only very simple
combinatorial arguments \cite{schwartz}, \cite{zippel}. To answer
partly 
question Q2  in
terms of the total degree Pellikaan and Wu in~\cite{pw_expanded} followed the footprint
bound approach. Later a generalized Schwartz-Zippel bound 
that  deals
with question Q1  in terms of the total degree was suggested by Augot,
El-Khamy, McEliece, Parvaresh, Stepanov, Vardy 
in~\cite{manyauthors} for the case of two variables, and by Augot,
Stepanov in \cite{augot} for arbitrarily many variables. The bound was
proven to be correct in a recent paper by  Dvir, Kopparty, Saraf and Sudan~\cite{dvir}. The generalized Schwartz-Zippel bound goes as follows.
\begin{Theorem}\label{SZ-lemma}
Let $F(\vec{X}) \in {\mathbf{F}}[\vec{X}]$ be a non-zero polynomial of
total degree $d$. Then for any finite set $S \subseteq {\mathbf{F}}$
\begin{eqnarray}
\sum_{\vec{a}\in S^n}{\mbox{mult}}(F,\vec{a}) \leq d | S |^{{m-1}}. \nonumber
\end{eqnarray}
\end{Theorem}
As a corollary we get an immediate partial answer to question Q2 in
terms of the total degree of $F$.
\begin{Corollary}\label{corsz}
Let $F(\vec{X}) \in {\mathbf{F}}[\vec{X}]$ be a
non-zero polynomial of total degree $d$ and let $S \subseteq
{\mathbf{F}}$ be finite.
The number
of zeros of $F$ of multiplicity at least $r$ from $S^m$ is at most
$$\frac{d}{r}|S |^{m-1}.$$
\end{Corollary}
In the present paper we take the Schwartz-Zippel
approach. We use the methods from~\cite{dvir}, but rather than
taking into account only information about the total degree and
allowing only point ensembles $S^n$ we
\begin{itemize}
\item use information about the leading monomial with respect to a
  lexicographic ordering.
\item consider the more general point ensembles $S_1\times \cdots
  \times S_m$, the sets $S_i$ all being 
finite.
\end{itemize}
In Section~\ref{secleading} we easily translate Theorem~\ref{SZ-lemma} into this setting and
derive an immediate translation of Corollary~\ref{corsz}. As will be
shown in Section~\ref{seclin},  
Theorem~\ref{SZ-lemma} and its translation are
tight for all products of univariate linear
terms. A similar result by no means holds for
Corollary~\ref{corsz} and its translation. Actually, a refinement of the methods 
from~\cite{dvir} yields for dramatic improvements to
Corollary~\ref{corsz} and its translation. In its most general
form in Section~\ref{sec-cor-sz-gen-bet} we state 
an algorithm to upper bound the number of zeros of multiplicity at
least $r$. Using this algorithm we then derive in Section~\ref{sectwovar}
closed formulas in the case where the number of variables is two and
the multiplicity is arbitrary. Section~\ref{secismall} further presents a
simple closed formula for the case of arbitrary many variables where, however, the
powers $i_1, \ldots , i_m$ in the leading monomial
${\mbox{lm}}(F)=X_1^{i_1}\cdots X_m^{i_m}$ are all small. In
Section~\ref{seclin} we consider the case when the polynomial is a
product of univariate linear terms. Such polynomials are easy to
analyze and by doing this we get in appendix~\ref{applin} an algorithm to
produce lower bounds on the maximal attainable number of zeros of
multiplicity at least $r$. Section~\ref{secequal} describes various
conditions under which our upper bound equals our lower bound. Having improved on
the results in~\cite[Section 2]{dvir} we conclude the paper by showing 
in Appendix~\ref{seccomp} that
Corollary~\ref{corsz} is stronger than the corresponding result
given by Pellikaan and Wu in~\cite{pw_expanded}. From this we can conclude that the results found in the
present paper are the strongest known. The present paper comes with a
webpage~\cite{hmpage} where a large number of experimental results
are presented.
\section{Using information about the leading monomial}\label{secleading}
In the following we modify the method from~\cite[Section 2]{dvir}. One could choose to prove the results of the present
section using the original method, however, the modification will
be needed in the section to follow. For simplicity we stick to the
modified method in both sections. Throughout the paper
$S_1,\ldots, S_m \subseteq {\mathbf{F}}$ are finite sets and we write $s_1=|S_1 | , \ldots , s_m=|S_m|$. In the following the monomial
ordering $\prec$ on the set of monomials in variables $X_1,
\ldots , X_m$ will always be the lexicographic ordering with $X_m\prec \cdots \prec X_1$.\\

We start our investigations by recalling two results from~\cite[Section
2]{dvir}. The first corresponds to~\cite[Lemma 5]{dvir}.
\begin{Lemma}\label{lem5}
Consider $F(\vec{X}) \in {\mathbf{F}}[\vec{X}]$ and
$\vec{a} \in {\mathbf{F}}^m$. For any $\vec{k}=(k_1, \ldots ,k_m)
\in {\mathbf{N}}_0^m$ we have 
$${\mbox{mult}}(F^{(\vec{k})},\vec{a}) \geq mult(F,\vec{a})-(k_1+ \cdots +k_m).$$
\end{Lemma}
The next result that we recall corresponds to the last part of
\cite[Proposition 6]{dvir}. 
\begin{Proposition}\label{propsammensat}
 Given $F(X_1, \ldots , X_m) \in {\mathbf{F}}[X_1, \ldots
 ,X_m]$ and
$$Q(Y_1, \ldots , Y_l)=(Q_1(\vec{Y}), \ldots  ,Q_m(\vec{Y}))
\in {\mathbf{F}}[Y_1, \ldots , Y_l]^m$$ let $F \circ Q$ be the
polynomial $F(Q_1(\vec{Y}), \ldots , Q_m(\vec{Y}))$.  For any
$\vec{a}\in {\mathbf{F}}^l$ we have $${\mbox{mult}}(F \circ Q,\vec{a})
\geq {\mbox{mult}}(F,Q(\vec{a})).$$
\end{Proposition}
We get the following Corollary, which is closely related
to~\cite[Corollary 7]{dvir}. 
\begin{Corollary}\label{cor8}
Let $F(X_1, \ldots , X_m) \in {\mathbf{F}}[X_1, \ldots , X_m]$ and
$
\vec{b}_1, \ldots  ,\vec{b}_{m-1},\vec{c} \in {\mathbf{F}}^m$ be
given. Write $F^\ast(T_1, \ldots  ,T_{m-1}) =F(T_1\vec{b}_1 + \cdots +
T_{m-1}\vec{b}_{m-1}+\vec{c})$.
For any $(t_1, \ldots , t_{m-1}) \in {\mathbf{F}}^{m-1}$ we
have
\begin{multline*}
{\mbox{mult}}(F^\ast(T_1, \ldots , T_{m-1}),(t_1, \ldots , t_{m-1})) \\
\geq{\mbox{mult}}(F(X_1, \ldots ,X_m),
t_1\vec{b}_1+\cdots +t_{m-1}\vec{b}_{m-1}+\vec{c}).
\end{multline*}
\end{Corollary}
We now write
$$F(X_1, \ldots , X_m)=\sum_{j_1, \ldots  ,j_{m-1}} 
X_1^{j_1}\cdots X_{m-1}^{j_{m-1}}F_{j_1, \ldots
  j_{m-1}} (X_m).$$
Let $X_1^{i_1} \cdots X_m^{i_m}$ be the leading monomial of $F$ with
respect to $\prec$. Then
due to the definition of $\prec$, $F_{i_1, \ldots , i_{m-1}}(X_m)$ is a (univariate) polynomial
of degree $i_m$.
For $a_m \in {\mathbf{F}}$  define  
$$
r(a_m)={\mbox{mult}}(F_{i_1, \ldots , i_{m-1}}(X_m),a_m).$$
Clearly, 
\begin{equation}
\sum_{a_m \in S_m}r(a_m) \leq i_m. \label{eqfoerste}
\end{equation}
We have 
$$
F^{(0,\ldots , 0,r(a_m))}(X_1, \ldots , X_m)=\sum_{j_1, \ldots ,
  j_{m-1}}
X_1^{j_1} \cdots X_{m-1}^{j_{m-1}}F_{j_1, \ldots
  ,j_{m-1}}^{(r(a_m))}(X_m)$$
and due to the definition of $\prec$ and to the definition of
$r(a_m)$ we have
\begin{equation}
{\mbox{lm}}_{\prec}(F^{(0,\ldots , 0,r(a_m))}(X_1, \ldots ,
X_{m-1},a_m))=X_1^{i_1}\cdots X_{m-1}^{i_{m-1}}. \label{eqnaeste}
\end{equation}
Applying first Lemma~\ref{lem5} with $\vec{k}=(0,\ldots ,0,r(a_m))$ and 
afterwards   
Corollary~\ref{cor8} with 
$\vec{b}_1=(1,0,\ldots ,0), \ldots ,
\vec{b}_{m-1}=(0, \ldots , 0,1,0)$, $\vec{c}=(0, \ldots , 0,a_m)$ and $t_1=a_1, \ldots ,
t_{m-1}=a_{m-1}$ we get the following result which is
closely related to a result in~\cite[Proof of Lemma 8]{dvir}:
\begin{align}
&{\mbox{mult}}\big(F(X_1, \ldots  ,X_m),(a_1, \ldots , a_m)\big)\nonumber \\
&\leq (0+ \cdots +0+r(a_m))+{\mbox{mult}}\big(F^{(0,\ldots
  ,0,r(a_m))}(X_1, \ldots ,X_m),(a_1, \ldots , a_m)\big)\nonumber\\
&\leq r(a_m) +{\mbox{mult}}\big(F^{(0,\ldots  ,0,r(a_m))}(X_1, \ldots
,X_{m-1},a_m),(a_1, \ldots  ,a_{m-1})\big). \label{eqtredie}
\end{align} 
We are now in the position that we can prove the main result of this section.
\begin{Theorem}\label{prop-sz-gen}
Let $F(\vec{X}) \in {\mathbf{F}}[\vec{X}]$ be a non-zero polynomial and
let ${\mbox{lm}}(F)=X_1^{i_1} \cdots X_m^{i_m}$ be its leading
monomial with respect to a lexicographic ordering. Then for
any finite sets $S_1, \ldots ,S_m \subseteq {\mathbf{F}}$
\begin{eqnarray}
\sum_{\vec{a}\in S_1 \times \cdots \times S_m}{\mbox{mult}}(F,\vec{a})
\leq i_1s_2\cdots s_m+s_1i_2s_3 \cdots s_m+\cdots +s_1\cdots s_{m-1}i_m.\nonumber
\end{eqnarray}
\end{Theorem}
\begin{proof}
We prove the theorem for the monomial
ordering $\prec$. Dealing with general 
lexicographic orderings is simply a question of relabeling the
variables. Clearly the theorem holds for
$m=1$. For $m>1$ we consider~(\ref{eqtredie}). Assuming the
theorem holds when the number of variables is smaller than $m$ we
get by applying~(\ref{eqfoerste}) and~(\ref{eqnaeste}) the following estimate
\begin{align}
&\sum_{\vec{a} \in S_1 \times \cdots \times
  S_m}{\mbox{mult}}(F,\vec{a}) \nonumber \\
&\leq i_ms_1 \cdots s_{m-1}+s_m(i_1s_2
\cdots s_{m-1}+\cdots +i_{m-1}s_1\cdots s_{m-2})\nonumber \\
&=i_1s_2\cdots s_m+i_2s_1 s_3\cdots s_m+\cdots i_ms_1\cdots
s_{m-1}\nonumber
\end{align}
as required.
\end{proof}

We have the following immediate generalization of Corollary~\ref{corsz}.
\begin{Corollary}\label{cor-sz-gen}
Let $F(\vec{X}) \in {\mathbf{F}}[\vec{X}]$ be a non-zero polynomial and
let ${\mbox{lm}}(F)=X_1^{i_1} \cdots X_m^{i_m}$ be its leading
monomial with respect to a lexicographic ordering. Assume $S_1, \ldots
,S_m \subseteq {\mathbf{F}}$ are finite sets. 
Then over
$S_1 \times \cdots \times S_m $ the number
of zeros of multiplicity at least $r$ is less than or equal to 
\begin{eqnarray}
\big( i_1s_2\cdots s_m+s_1i_2s_3\cdots s_m+\cdots +s_1\cdots
s_{m-1}i_m \big) /r.\nonumber
\end{eqnarray}
\end{Corollary}

{\section{Improvements to Corollary~\ref{cor-sz-gen} \label{sec-cor-sz-gen-bet}
}}
In this section we shall see that a further analysis allows for
dramatic improvements to Corollary~\ref{cor-sz-gen}.
Let $X_1^{i_1}\cdots X_m^{i_m}$ be the leading monomial of $F(\vec{X})\in {\mathbf{F}}[\vec{X}]$  with respect
to $\prec$. Recall from~(\ref{eqtredie}) the bound
\begin{align}
&{\mbox{mult}}(F(X_1, \ldots  ,X_m),(a_1, \ldots , a_m))\nonumber \\
& \leq r(a_m) +{\mbox{mult}}(F^{(0,\ldots  ,0,r(a_m))}(X_1, \ldots
,X_{m-1},a_m),(a_1, \ldots  ,a_{m-1})).\label{eqnoget}
\end{align}
Here, $r(a_m)$ are numbers that when summed over all possible $a_m \in S_m$
give at most $i_m$ and the leading monomial of $F^{(0,\ldots  ,0,r(a_m))}(X_1, \ldots
,X_{m-1},a_m)$ with respect to $\prec$ is $X^{i_1} \cdots
X_{m-1}^{i_{m-1}}$. Our analysis suggests the
following recursive definition of a function to bound the number of
zeros of multiplicity $r$.\\

\begin{Definition}\label{defD}
Let $r \in {\mathbf{N}}, i_1, \ldots , i_m \in {\mathbf{N}}_0$. Define 
$$D(i_1,r,s_1)=\min \big\{\big\lfloor \frac{i_1}{r} \big\rfloor,s_1\big\}$$
and for $m \geq 2$
\begin{multline*}
D(i_1, \ldots , i_m,r,s_1, \ldots ,s_m)=
\\
\begin{split}
\max_{(u_1, \ldots  ,u_r)\in A(i_m,r,s_m) }&\bigg\{ (s_m-u_1-\cdots -u_r)D(i_1,\ldots ,i_{m-1},r,s_1,
\ldots ,s_{m-1})\\
&\quad+u_1D(i_1, \ldots , i_{m-1},r-1,s_1, \ldots ,s_{m-1})+\cdots
\\
&\quad +u_{r-1}D(i_1, \ldots ,i_{m-1},1,s_1, \ldots , s_{m-1})+u_rs_1\cdots
s_{m-1} \bigg\}
\end{split}
\end{multline*}
where 
\begin{multline}
A(i_m,r,s_m)= \nonumber \\
\{ (u_1, \ldots , u_r) \in {\mathbf{N}}_0^r \mid u_1+ \cdots
+u_r \leq s_m {\mbox{ \ and \ }} u_1+2u_2+\cdots +ru_r \leq i_m\}.\nonumber
\end{multline}
\end{Definition}
\begin{Theorem}\label{prorec}
For a polynomial $F(\vec{X})\in {\mathbf{F}}[\vec{X}]$ let $X_1^{i_1}\cdots X_m^{i_m}$ be its leading monomial with
respect to $\prec$ (this is the lexicographic ordering with $X_m\prec \cdots \prec X_1$). Then $F$ has at most $D(i_1, \ldots , i_m,r,s_1,
\ldots ,s_m)$ zeros of multiplicity at least $r$ in $S_1\times \cdots
\times S_m$. The corresponding recursive algorithm produces a number
that is at most equal to the number found in
Corollary~\ref{cor-sz-gen} and is at most equal to $s_1 \cdots s_m$.
\end{Theorem}
\begin{proof}
The proof of the first part of the proposition is an induction proof. The result clearly holds for
$m=1$. Given $m>1$ assume it holds for $m-1$. 
For $d=1, \ldots ,
r-1$ let $u_d$ be the number of $a_m$'s with $r(a_m)=d$ and let $u_r$
be the number of $a_m$'s with $r(a_m) \geq r$. The number of $a_m$'s
with $r(a_m)=0$ is $s_m-u_1-\cdots -u_r$. The boundary conditions that
$u_1+\cdots +u_r \leq s_m$ and $u_1+2u_2+\cdots +ru_r \leq i_m$ are
obvious. For every $a_m$ with $r(a_m)=d$, $d=0, \ldots , r-1$ for
$(a_1, \ldots , a_m)$ to be a zero of multiplicity at least $r$ the
last expression in~(\ref{eqnoget}) must be at least $r-d$. For
$a_m$ with $r(a_m)\geq r$ all choices of $a_1, \ldots , a_{m-1}$ are
legal. This proves the first part of the proposition. As both
Corollary~\ref{cor-sz-gen} and the above proof rely on~(\ref{eqnoget}), Theorem~\ref{prorec} cannot produce a
number greater than what is found in Corollary~\ref{cor-sz-gen}. The
condition $u_1+\cdots +u_r \leq s_m$ and the definition of
$D(i_1,r,s_1)$ imply the last result.
\end{proof}

The next remark shows that we need only apply the algorithm to a
restricted set of exponents $(i_1,\ldots ,i_m)$.

\begin{Remark}\label{rembig}
Given $(i_1, \ldots , i_m, r, s_1, \ldots ,s_m)$ with $\lfloor
i_1/s_1\rfloor+\cdots +\lfloor i_m/s_m\rfloor \geq r$ then there exist
polynomials with the leading monomial being $X_1^{i_1} \cdots X_m^{i_m}$ such that all points in $S_1 \times \cdots \times S_m$ are
zeros of multiplicity at least $r$. 
Hence, we need only apply
the algorithm to cases that do not satisfy the above inequality. In Section~\ref{seclin}, Example~\ref{exbig}, we will
explain this fact in more detail. 
\end{Remark}
In a series of experiments we found that the above algorithm produces
numbers that are often much lower than the minimum of the corresponding result from
Corollary~\ref{cor-sz-gen} and $s_1\cdots s_m$. In the
webpage~\cite{hmpage} we list all results of our experiments. Here,
we only mention a
few.
\begin{Example}\label{exny1}
In this example we bound the number of zeros of multiplicity 
$3$ or more for polynomials in two
variables. Both $S_1$ and $S_2$ are assumed to be of size $5$. Table~\ref{tabny1} shows 
information obtained from our algorithm for the exponents $i_1,i_2$
not treated by Remark~\ref{rembig}. Table~\ref{tabny2} illustrates the
improvement on the  bound $\lfloor \min\{
(i_1+i_2)5/3,5^2\}\rfloor$. Here, the first expression comes from
Corollary~\ref{cor-sz-gen} and the last expression is the number of
points in $S_1 \times S_2$. Observe, that the tables are not symmetric
meaning that $D(i_1,i_2,3,5,5)$ does not always equal $D(i_2,i_1,3,5,5)$.
\begin{table}
\centering
\caption{$D(i_1,i_2,3,5,5)$}
\newcommand{\SP}{~~}
\begin{tabular}{@{}c@{}cc@{\SP}c@{\SP}c@{\SP}c@{\SP}c@{\SP}c@{\SP}c@{\SP}c@{\SP}c@{\SP}c@{\SP}c@{\SP}c@{\SP}c@{\SP}c@{\SP}c@{}}
    \toprule
    &&\multicolumn{15}{c}{$i_1$}\\
    &&0   &1   &2   &3   &4   &5   &6   &7   &8   &9   &10  &11  &12  &13  &14\\
    \addlinespace
    \multirow{15}{*}{$i_2$}
    &\multicolumn{1}{c}{0 }&0 &0 &0 &5 &5 &5 &10&10&10&15&15&15&20&20&20\\
    &\multicolumn{1}{c}{1 }&0 &0 &1 &5 &6 &6 &11&11&12&16&17&17&21&21&21\\
    &\multicolumn{1}{c}{2 }&0 &1 &2 &7 &8 &9 &13&13&14&17&19&19&22&22&22\\
    &\multicolumn{1}{c}{3 }&5 &5 &5 &9 &9 &10&14&14&16&18&21&21&23&23&23\\
    &\multicolumn{1}{c}{4 }&5 &5 &6 &9 &11&13&16&16&18&19&23&23&24&24&24\\
    &\multicolumn{1}{c}{5 }&5 &6 &7 &11&12&14&17&17&20&20\\
    &\multicolumn{1}{c}{6 }&10&10&10&13&14&17&19&19&21&21\\
    &\multicolumn{1}{c}{7 }&10&10&11&13&15&18&20&20&22&22\\
    &\multicolumn{1}{c}{8 }&10&11&12&15&17&21&22&22&23&23\\
    &\multicolumn{1}{c}{9 }&15&15&15&17&18&22&23&23&24&24\\
    &\multicolumn{1}{c}{10}&15&15&16&17&20\\
    &\multicolumn{1}{c}{11}&15&16&17&19&21\\
    &\multicolumn{1}{c}{12}&20&20&20&21&22\\
    &\multicolumn{1}{c}{13}&20&20&21&21&23\\
    &\multicolumn{1}{c}{14}&20&21&22&23&24\\
    \bottomrule
\end{tabular}
\label{tabny1}
\end{table}

\begin{table}
\centering
\caption{Improvements found in Example~\ref{exny1}}
\newcommand{\SP}{~\hspace*{0.851em}}
\begin{tabular}{@{}c@{}cc@{\SP}c@{\SP}c@{\SP}c@{\SP}c@{\SP}c@{\SP}c@{\SP}c@{\SP}c@{\SP}c@{~~~}c@{~~}c@{~~}c@{~~}c@{~~}c@{}}
    \toprule
    &&\multicolumn{15}{c}{$i_1$}\\
    &&0&1&2&3&4&5&6&7&8&9&10&11&12&13&14\\
    \addlinespace
    \multirow{15}{*}{$i_2$}
    &\multicolumn{1}{c}{0}&0&1&3&0&1&3&0&1&3&0&1&3&0&1&3\\
    &\multicolumn{1}{c}{1}&1&3&4&1&2&4&0&2&3&0&1&3&0&2&4\\
    &\multicolumn{1}{c}{2}&3&4&4&1&2&2&0&2&2&1&1&2&1&3&3\\
    &\multicolumn{1}{c}{3}&0&1&3&1&2&3&1&2&2&2&0&2&2&2&2\\
    &\multicolumn{1}{c}{4}&1&3&4&2&2&2&0&2&2&2&0&2&1&1&1\\
    &\multicolumn{1}{c}{5}&3&4&4&2&3&2&1&3&1&3\\
    &\multicolumn{1}{c}{6}&0&1&3&2&2&1&1&2&2&4\\
    &\multicolumn{1}{c}{7}&1&3&4&3&3&2&1&3&3&3\\
    &\multicolumn{1}{c}{8}&3&4&4&3&3&0&1&3&2&2\\
    &\multicolumn{1}{c}{9}&0&1&3&3&3&1&2&2&1&1\\
    &\multicolumn{1}{c}{10}&1&3&4&4&3\\
    &\multicolumn{1}{c}{11}&3&4&4&4&4\\
    &\multicolumn{1}{c}{12}&0&1&3&4&3\\                                       
    &\multicolumn{1}{c}{13}&1&3&4&4&2\\
    &\multicolumn{1}{c}{14}&3&4&3&2&1\\
    \bottomrule
\end{tabular}
\label{tabny2}
\end{table}
\end{Example}
\newpage
\begin{Example}\label{exny2}

In this example we bound the number of zeros of multiplicity 
$3$ or more for polynomials in four
variables. The sets $S_1$, $S_2$, $S_3$ and $S_4$ are all assumed to
be of size $6$. Table~\ref{tabny3} shows 
information obtained from our algorithm for a small sample of values
$(i_1,i_2,i_3=3,i_4=5)$. Table~\ref{tabny4} illustrates the
improvement on the  bound $\min\{
(i_1+i_2+i_3+i_4)6^3/3,6^4\}$. Here, the first expression comes from
Corollary~\ref{cor-sz-gen} and the last expression is the number of
points in $S_1 \times S_2\times S_3\times S_4$.
\begin{table}
\centering
\caption{$D(i_1,i_2,i_3=3,i_4=5,3,6,6,6,6)$}

\newcommand{\SP}{~\hspace*{0.851em}}
\begin{tabular}{@{}c@{~~}c@{\SP}c@{\SP}c@{\SP}c@{\SP}c@{\SP}c@{\SP}c@{\SP}c@{\SP}c@{}}
    \toprule
    & &\multicolumn{8}{c}{$i_1$} \\
    & &   0   &   1   &   2   &    3   &    4   &    5   &    6   &    7   \\
    \addlinespace
    \multirow{8}{*}{$i_2$}
    &0&468&486&504& 642& 666& 720& 912& 912\\
    &1&486&501&536& 651& 705& 764& 964& 964\\
    &2&504&536&574& 700& 759& 840&1024&1024\\
    &3&642&651&666& 771& 816& 908&1077&1077\\
    &4&666&684&732& 807& 880& 984&1140&1140\\
    &5&720&750&816& 876& 952&1056&1197&1197\\
    &6&912&912&960& 976&1024&1134&1260&1260\\
    &7&912&928&980&1008&1060&1155&1263&1263\\
    \bottomrule
\end{tabular}
\label{tabny3}
\end{table}

\begin{table}
\centering
\caption{Improvements found in Example~\ref{exny2}}
\newcommand{\SP}{~\hspace*{0.851em}}
\begin{tabular}{@{}c@{~~}c@{\SP}c@{\SP}c@{\SP}c@{\SP}c@{\SP}c@{\SP}c@{\SP}c@{\SP}c@{}}
    \toprule
    & &\multicolumn{8}{c}{$i_1$} \\
    & &   0   &   1   &   2   &    3   &    4   &    5   &    6   &    7   \\
    \addlinespace
    \multirow{8}{*}{$i_2$}
    &0&108&162&216&150&198&216&96 &168\\
    &1&162&219&256&213&231&244&116&188\\
    &2&216&256&290&236&249&240&128&200\\
    &3&150&213&270&237&264&244&147&219\\
    &4&198&252&276&273&272&240&156&156\\
    &5&216&258&264&276&272&240&99 &99 \\
    &6&96 &168&192&248&272&162&36 &36 \\
    &7&168&224&244&288&236&141&33 &33 \\
    \bottomrule
\end{tabular}
\label{tabny4}
\end{table}
\end{Example}
\newpage
\begin{Example}\label{exangle}
Let $s_1=\cdots =s_m=q$. Our experiments listed in~\cite{hmpage} show that the value $D(i_1, \ldots ,
i_m,r,q,\ldots ,q)$
often improves dramatically on the previous known bounds. We here list the
maximal attained improvement for a selection of fixed values of $m, q,
r$. We do this relatively to the number of points in $S_1 \times
\cdots \times S_m$. In other words we list in Table~\ref{tabny5} the value
$$\bigg(\max_{i_1, \ldots , i_m} \{\min\{ (i_1+\cdots
i_m)q^{m-1}/r,q^m\}-D(i_1, \ldots , i_m,r,q,\ldots ,q)\}\bigg)/q^m$$
for various choices of $m, q, r$. 
\begin{table}
\centering
\caption{Maximum improvements relative to $q^m$; truncated}
\newcommand{\SP}{~~}
\begin{tabular}{@{}c@{}c@{~~~~}r@{.}l@{\SP}r@{.}l@{\SP}r@{.}l@{\SP}r@{.}l@{\SP}r@{.}l@{\SP}r@{.}l@{\SP}r@{.}l@{\SP}r@{.}l@{\SP}r@{.}l@{\SP}r@{.}l@{}}
\toprule
$m$&& \multicolumn{8}{c}{2} & \multicolumn{8}{c}{3} & \multicolumn{4}{c}{4} \\
\cmidrule(r){3-10} \cmidrule(r){11-18} \cmidrule{19-22}
$r$&&\multicolumn{2}{l}{2}&\multicolumn{2}{l}{3}&\multicolumn{2}{l}{4}&\multicolumn{2}{l}{5}&\multicolumn{2}{l}{2}&\multicolumn{2}{l}{3}&\multicolumn{2}{l}{4}&\multicolumn{2}{l}{5}&\multicolumn{2}{l}{2}&\multicolumn{2}{l}{3}\\
\addlinespace
\multirow{7}{*}{$q$}
&\multicolumn{1}{c}{2}&0&25 &0&25 &0&25 &0&25 &0&25 &0&375&0&375&0&375&0&312&0&375\\
&\multicolumn{1}{c}{3}&0&222&0&222&0&222&0&222&0&296&0&296&0&296&0&296&0&296&0&333\\
&\multicolumn{1}{c}{4}&0&187&0&187&0&187&0&187&0&281&0&25 &0&25 &0&265&0&316&0&289\\
&\multicolumn{1}{c}{5}&0&24 &0&16 &0&16 &0&2  &0&256&0&256&0&232&0&24 &0&307&0&288\\
&\multicolumn{1}{c}{6}&0&222&0&194&0&166&0&166&0&277&0&25 &0&231&0&212&0&293&0&287\\
&\multicolumn{1}{c}{7}&0&204&0&204&0&163&0&142&0&279&0&244&0&227&0&209&0&299&0&276\\
&\multicolumn{1}{c}{8}&0&234&0&203&0&171&0&140&0&275&0&25 &0&214&\multicolumn{2}{l}{?}&0&299&0&275\\  
\bottomrule
\end{tabular}
\label{tabny5}
\end{table}
The experiments also show a distinct average improvement. This is
illustrated in Table~\ref{tabendnuenny1} where for fixed $q, r, m$ we list
the mean value of
\begin{equation}
\frac{ \min \{(i_1+\cdots + i_m)q^{m-1},q^m\}-D(i_1, \ldots ,
i_m,r,q, \ldots ,q)}{\min \{(i_1+\cdots +
i_m)q^{m-1},q^m\}}.\label{eqangle}
\end{equation}
The average is taken over the set of exponents $(i_1, \ldots ,
i_m)\neq \vec{0}$ where 
$\lfloor i_1/q\rfloor+\cdots +\lfloor i_m/q\rfloor
< r$  
holds.
\begin{table}[!h]
\centering
\caption{The mean value of (\ref{eqangle}); truncated}
\newcommand{\SP}{~~}
\begin{tabular}{@{}c@{}c@{~~~~}r@{.}l@{\SP}r@{.}l@{\SP}r@{.}l@{\SP}r@{.}l@{\SP}r@{.}l@{\SP}r@{.}l@{\SP}r@{.}l@{\SP}r@{.}l@{\SP}r@{.}l@{\SP}r@{.}l@{}}
\toprule
$m$&& \multicolumn{8}{c}{2} & \multicolumn{8}{c}{3} & \multicolumn{4}{c}{4} \\
\cmidrule(r){3-10} \cmidrule(r){11-18} \cmidrule{19-22}
$r$&&\multicolumn{2}{l}{2}&\multicolumn{2}{l}{3}&\multicolumn{2}{l}{4}&\multicolumn{2}{l}{5}&\multicolumn{2}{l}{2}&\multicolumn{2}{l}{3}&\multicolumn{2}{l}{4}&\multicolumn{2}{l}{5}&\multicolumn{2}{l}{2}&\multicolumn{2}{l}{3}\\
\addlinespace
\multirow{7}{*}{$q$}
&\multicolumn{1}{c}{2}&0&363&0&273&0&337&0&291&0&301&0&300&0&342&0&307&0&248&0&260\\
&\multicolumn{1}{c}{3}&0&217&0&286&0&228&0&236&0&194&0&224&0&213&0&214&0&158&0&177\\
&\multicolumn{1}{c}{4}&0&191&0&197&0&232&0&195&0&158&0&169&0&180&0&172&0&125&0&135\\
&\multicolumn{1}{c}{5}&0&155&0&167&0&174&0&197&0&139&0&145&0&148&0&153&0&110&0&116\\
&\multicolumn{1}{c}{6}&0&148&0&160&0&156&0&154&0&128&0&132&0&132&0&131&0&100&0&105\\
&\multicolumn{1}{c}{7}&0&128&0&137&0&138&0&138&0&119&0&122&0&121&0&119&0&093&0&098\\
&\multicolumn{1}{c}{8}&0&126&0&127&0&134&0&126&0&114&0&115&0&113&\multicolumn{2}{l}{?}&0&089&0&093\\
\bottomrule
\end{tabular}
\label{tabendnuenny1}
\end{table}
\end{Example}

\begin{Example}
In Example~\ref{exny1} for any total degree $d$ there exists a choice of
$i_1,\ldots , i_m$ with $i_1+\cdots +i_m=d$ such that 
$$D(i_1,\cdots , i_m,r,q,\ldots , q)=\min\{ dq^{m-1}/r,q^m\}.$$
However, there are cases where such a result does not hold. Going through
all possible choices of $i_1,i_2,i_3$ with $i_1+i_2+i_3=12$ we see
that the largest obtained value of $D(i_1,i_2,i_3,3,8,8,8)$ equals
$224$ whereas $\min\{12\cdot 8^2/3,8^3\}=256$.
\end{Example}
The next two examples are of a theoretical nature.
\begin{Example}\label{ex1}
Given an arbitrary monomial ordering let ${\mbox{lm}}(F)=X_1^{i_1} \cdots X_m^{i_m}$ with $i_1 \leq s_1, \ldots , i_m
\leq s_m$. Using results from Gr\"{o}bner
basis theory we can deduce that $F$ can have no more than
\begin{equation}
s_1\cdots s_m-(s_1-i_1)\cdots (s_m-i_m) \label{eqrlig1}
\end{equation}
zeros (of multiplicity $1$ or more) over $S_1 \times \cdots \times S_m$. (See~\cite{hyperbolic} and
\cite{secondweight} for the case of $S_1=\cdots=S_m={\mathbf{F}}_q$.)
This result is known to be sharp meaning that polynomials exist with
this many zeros.
It is interesting to observe that~(\ref{eqrlig1}) follows as
an immediate corollary to Theorem~\ref{prorec} in the case where 
the monomial ordering $\prec$ is the pure lexicographic ordering with
$X_m\prec \cdots \prec X_1$. In contrast~(\ref{eqrlig1})
only equals the result in Corollary~\ref{cor-sz-gen} when
${\mbox{lm}}(F)$ is univariate; in general the
two bounds can differ very much. In
Section~\ref{secismall} we will see that for the case of the monomial
ordering being $\prec$ 
(\ref{eqrlig1}) can be viewed as a special case
of a more general result.
\end{Example}
\begin{Example}
Consider that the leading monomial is
univariate, i.e.\ ${\mbox{lm}}(F)=X_t^{i_t}$ for some $t \in \{1, \ldots
, m\}$. Theorem~\ref{prorec} tells us that $F$ can have at most 
$$s_1 \cdots s_{t-1} \lfloor \frac{i_t}{r}\rfloor   s_{t+1} \cdots s_m$$
zeros of multiplicity $r$ or more. In contrast Corollary~\ref{cor-sz-gen} only
gives us the bound 
$$\lfloor s_1\cdots s_{t-1} \frac{i_t}{r}s_{t+1}\cdots s_m \rfloor.$$
For $m >1$ the bounds are the same only when $r$ divides $i_t$. Assume in
larger generality that $i_{t_1}, \ldots ,i_{t_v}$, $t_u < t_w$ for $u
< w$ are the non-zero elements in $\{i_1, \ldots , i_m\}$. Then
$$D(i_1, \ldots , i_m,r,s_1, \ldots ,s_m)=
\bigg( \prod_{i_d=0}s_d \bigg) D(i_{t_1},\ldots,i_{t_v},r,s_{t_1}, \ldots , s_{t_v}).$$ 
\end{Example}

{\section{The case of two variables}\label{sectwovar}}
\noindent In this section we derive closed formulas for the case of two variables and
the multiplicity being arbitrary.  By
Remark~\ref{rembig} the following Proposition covers all non-trivial cases.\\
\begin{Proposition}\label{protwovar}
For $k=1, \ldots , r-1$,  $D(i_1,i_2,r,s_1,s_2)$ is upper bounded by\\
$\begin{array}{cl}
{\mbox{(C.1)}}&  {\displaystyle{s_2\frac{i_1}{r}+\frac{i_2}{r}\frac{i_1}{r-k}}}\\
&{\mbox{if \  }}(r-k)\frac{r}{r+1}s_1 \leq i_1 < (r-k)s_1
{\mbox{ \ and \ }} 0\leq i_2 <ks_2\\
{\mbox{(C.2)}}&
  {\displaystyle{s_2\frac{i_1}{r}+((k+1)s_2-i_2)(\frac{i_1}{r-k}-\frac{i_1}{r})+(i_2-ks_2)(s_1-\frac{i_1}{r})}}\\
& {\mbox{if \ }}(r-k)\frac{r}{r+1}s_1 \leq i_1 < (r-k)s_1 {\mbox{ \
    and \ }} ks_2\leq i_2 <(k+1)s_2\\
{\mbox{(C.3)}}&
{\displaystyle{s_2\frac{i_1}{r}+\frac{i_2}{k+1}(s_1-\frac{i_1}{r})}}\\
&{\mbox{if \ }} (r-k-1)s_1 \leq i_1 < (r-k)\frac{r}{r+1}s_1 {\mbox{ \
    and \ }} 0 \leq i_2 < (k+1)s_2.
\end{array}
$\\
Finally,\\
$\begin{array}{cl}
{\mbox{(C.4)}}& {\displaystyle{D(i_1,i_2,r,s_1,s_2)=s_2\lfloor \frac{i_1}{r} \rfloor
  +i_2(s_1-\lfloor \frac{i_1}{r} \rfloor )}}\\
& {\mbox{if \ }} s_1(r-1) \leq i_1 < s_1r {\mbox{ \ and \ }} 0 \leq i_2 < s_2.
\end{array}
$\\
The above numbers are at most equal to $\min\{(i_1s_2+s_1i_2)/r, s_1s_2 \}$.
\end{Proposition}
\begin{proof}First we consider the values of $i_1, i_2,
r,s_1,s_2$ corresponding to one of the cases (C.1), (C.2), (C.3). Let
$k$ be the largest number (as in Proposition~\ref{protwovar})
  such that $i_1 < (r-k)s_1$. Indeed $k \in \{1, \ldots ,
r-1\}$. We have
\begin{multline}
\label{eqsnabelen}
D(i_1,i_2,r,s_1,s_2) \leq\\
\max_{(u_1, \ldots  ,u_r)\in B(i_2,r,s_2)} \bigg\{
s_2\frac{i_1}{r}+u_1(\frac{i_1}{r-1}-\frac{i_1}{r})+\cdots
+u_k(\frac{i_1}{r-k}-\frac{i_1}{r})\\
+u_{k+1}(s_1-\frac{i_1}{r})+ \cdots +u_r(s_1-\frac{i_1}{r})\bigg\}
\end{multline}
where 
\begin{multline*}
B(i_2,r,s_2)=\{(u_1, \ldots , u_r) \in {\mathbf{Q}}^r \mid 0 \leq u_1,
\ldots ,u_r,\\
u_1+\cdots +u_r \leq s_2, u_1+2u_2+
\cdots +ru_r \leq i_2\}.
\end{multline*}
We observe, that 
$$k(\frac{i_1}{r-l}-\frac{i_1}{r})\leq l
(\frac{i_1}{r-k}-\frac{i_1}{r})$$ 
holds for $l \leq k$. Furthermore, we have the biimplication
\begin{eqnarray}
(r-k)\frac{r}{r+1} s_1 
\leq i_1
\Leftrightarrow (k+1)(\frac{i_1}{r-k}-\frac{i_1}{r}) \geq
k(s_1-\frac{i_1}{r}). \nonumber
\end{eqnarray}
Therefore, if the conditions in (C.1) are satisfied then~(\ref{eqsnabelen}) takes on its maximum when
$u_k=\frac{i_2}{k}$ and the remaining $u_i$'s equal $0$. If the
conditions in (C.2) are satisfied then (\ref{eqsnabelen}) takes on its
maximum at $u_k=(k+1)s_2-i_2$, $u_{k+1}=(i_2-ks_2)$ and the remaining
$u_i$'s equal $0$. If the conditions in (C.3) are satisfied
then~(\ref{eqsnabelen}) takes on its maximal value at
$u_{k+1}=\frac{i_2}{k+1}$ and the remaining $u_i$'s equal $0$.\\
Finally, if $s_1(r-1) \leq i_1 < s_1r$ and $0 \leq i_2 \leq s_2$ then
$D(i_1,i_2,r,s_1,s_2)$ is the maximal value of
$$s_2\lfloor \frac{i_1}{r} \rfloor +u_1(s_1-\lfloor \frac{i_1}{r}
\rfloor)+ \cdots +u_r(s_1-\lfloor \frac{i_1}{r}
\rfloor)$$
over $B(i_2,r,s_2)$. The maximum is attained for $u_1=i_2$ and all other
$u_i$'s equal $0$. The proof of the last result follows the proof of
the last part of Theorem~\ref{prorec}.
\end{proof}

\begin{Remark}
Experiments show (see~\cite{hmpage}) that 
 the numbers produced by Proposition~\ref{protwovar} are often much
smaller than $\min \{ (i_1s_2 +s_1i_2)/r,s_1s_2\}$. However, there are
cases where they are the identical. This happens for example when
$i_1=s_1(r-1)$ and $r$ divides $s_1$ and $s_2$. In the
proof of (C.1), (C.2), (C.3) we allowed $u_1, \ldots , u_r$ to be
rational numbers rather than integers. Therefore we cannot expect the upper bounds in
Proposition~\ref{protwovar} to equal the true value of
$D(i_1,i_2,r,s_1,s_2)$ in general. Our experiments show that the
bounds in (C.1), (C.2), (C.3) are sometimes close to
$D(i_1,i_2,r,s_1,s_2)$ but not always. Hence the best information is
found by actually applying the algorithm from the previous section.
\end{Remark}
{\section{When $i_1, \ldots ,i_m$ are small}\label{secismall}}
\noindent
Having already four different cases when $m=2$ the situation gets
rather complicated when we have more variables. Assuming however that
all exponents $i_1, \ldots , i_m$ in the leading monomial are small we
can give a very simple formula. Whereas the formula is simple we must 
admit that the precise definition of $i_1, \ldots , i_m$ being small
is a little involved. It goes as follows.
\begin{Definition}\label{defconda}
Let $m \geq 2$. We say that $(i_1, \ldots , i_m,r,s_1, \ldots , s_m)$ satisfies Condition A if the following hold
$$
\begin{array}{rl}
(A.1)& i_1, \ldots ,i_m \leq s_m\\
(A.2)&s(s_1-\frac{i_1}{l}) \cdots (s_{m-2}-\frac{i_{m-2}}{l}) \leq l
  (s_1-\frac{i_1}{s})\cdots (s_{m-2}-\frac{i_{m-2}}{s})\\
& {\mbox{ \ for all \ }} 
  l=2, \ldots , r, s=1, \ldots l-1.\\
(A.3)&s(s_1-\frac{i_1}{r}) \cdots (s_{m-1}-\frac{i_{m-1}}{r}) \leq r
  (s_1-\frac{i_1}{s})\cdots (s_{m-1}-\frac{i_{m-1}}{s})\\
&{\mbox{ \  for all \ }} s=1, \ldots
  , r-1.
\end{array}
$$
\end{Definition} 
\begin{Example}\label{exny}
If $r=1$ then (A.2) and (A.3) do not apply and with a reference to
Remark~\ref{rembig} (A.1) is a natural requirement.
\end{Example}
\begin{Example}\label{exsmart}
For $m=2$ and $r$ arbitrary  condition (A.2) does not apply and
condition (A.3)
simplifies to
$$i_1 \leq \frac{r^2s-rs^2}{r^2-s^2}s_1$$
for all $s$ with $1 \leq s <  r$. The minimal upper bound on
$i_1$ is attained for $s=1$. Hence, in case of two variables Condition A reads $i_1
\leq \frac{r}{r+1}s_1$, $i_2 \leq s_2$. 
\end{Example}
\begin{Example}\label{exsmalll}
For $r=2$ conditions (A.2), (A.3) simplifies all together to  
$$\big(s_1-\frac{i_1}{2}\big) \cdots\big(s_{m-1}-\frac{i_{m-1}}{2}\big) \leq
2\big(s_1-i_1\big) \cdots\big(s_{m-1}-i_{m-1}\big).$$
For $r=2$, $m=3$ and $s_1=s_2=s_3=q$ Condition A therefore reads
$$\frac{3}{2}(I_1+I_2)-\frac{7}{4}I_1I_2 \leq 1, {\mbox{ \ \ }} I_3
\leq 1$$
where $I_1=i_1/q$, $I_2=i_2/q$ and $I_3=i_3/q$. For $r=2$, $m=4$ and $s_1=s_2=s_3=s_4=q$ Condition A reads
  $$\frac{3}{2}(I_1+I_2+I_3)-\frac{7}{4}(I_1I_2+I_1I_3+I_2I_3)+\frac{15}{8}I_1I_2I_3 \leq 1, {\mbox{ \ \ }} I_4 \leq 1$$
where $I_4=i_4/q$. This is illustrated in Figure~\ref{figimp}. 
\begin{figure}
\begin{center}
\input{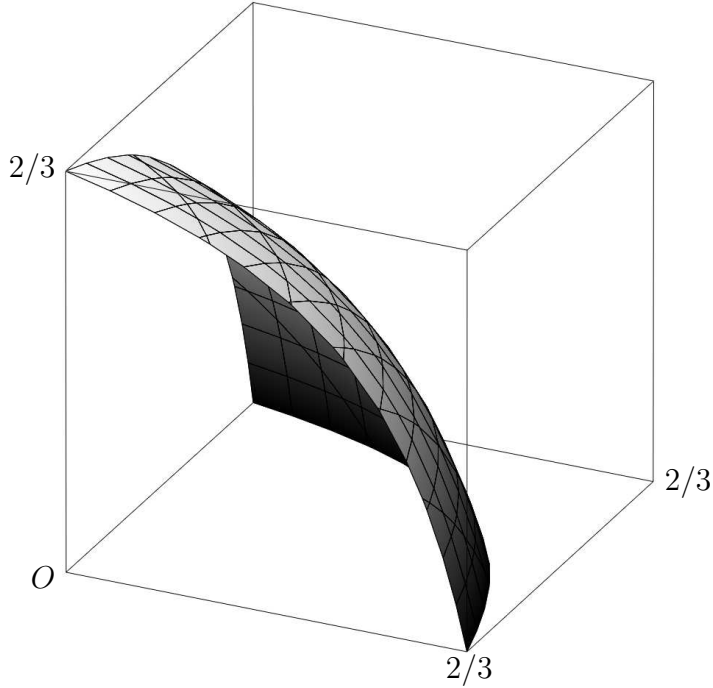}
\end{center}
\caption{The surface $\frac{3}{2}(I_1+I_2+I_3)-\frac{7}{4}(I_1I_2+I_1I_3+I_2I_3)+\frac{15}{8}I_1I_2I_3 = 1$}
\label{figimp}
\end{figure}
\end{Example}
In Example~\ref{ex1} we discussed a well known bound on the number of
zeros of 
multiplicity at least $r=1$. With Example~\ref{exny} in mind the last part of the following Proposition can be viewed as
a generalization of this bound. 
\begin{Proposition}\label{prosmall}
Assume that $(i_1, \ldots , i_m,r,s_1, \ldots , s_m)$, $m\geq 2$ satisfies Condition
A. If $r \geq 2$ then 
\begin{equation}
i_1 \leq \frac{r}{r+1}s_1, \ldots , i_{m-1} \leq \frac{r}{r+1}s_{m-1}.\label{eqeqeq}\end{equation}
For general $r$ we have 
\begin{eqnarray}
D(i_1, \ldots ,i_m,r,s_1,\ldots , s_m) \leq s_1\cdots s_m-(s_1-\frac{i_1}{r})\cdots
(s_m-\frac{i_m}{r})\label{eqendnuenstjerne}
\end{eqnarray}
which is at most equal to $\min \{ (i_1s_2\cdots s_m+\cdots +s_1\cdots
s_{m-1}i_m)/r, s_1\cdots s_m\}$.
\end{Proposition}
\begin{proof}
We start by noting that (A.2) implies
$$
(s_1-\frac{i_1}{l}) \cdots (s_{t-1}-\frac{i_{t-1}}{l}) \leq l
  (s_1-\frac{i_1}{s})\cdots (s_{t-1}-\frac{i_{t-1}}{s})
$$
for all 
  $t=2,\ldots , m-1$, $l=2, \ldots , r$, $s=1, \ldots l-1$. A similar
  thing holds regarding (A.3) and if we combine this fact with the
  result in Example~\ref{exsmart} we get 
(\ref{eqeqeq}) for $r \geq 2$.
Now let $(i_1, \ldots ,i_m,r,s_1, \ldots ,s_m)$ with $m>1$ be
such that Condition A holds. We give an induction proof
that 
\begin{equation}
\begin{array}{r}
D(i_1, \ldots , i_t,l,s_1,\ldots ,s_t)  \leq s_1\cdots s_t-(s_1-\frac{i_1}{l})
\cdots (s_t-\frac{i_{t}}{l})\\
 {\mbox{ \ for all \ }}1 \leq t < m , 1 \leq
  l\leq r
\end{array}
\label{eqminus1}
\end{equation} 
For $t=1$ the
result is clear. Let $1 < t<m$ and assume the result holds when $t$ is
substituted with $t-1$. According to Definition~\ref{defD} we have
\begin{multline*}
D(i_1, \ldots , i_t,l,s_1, \ldots ,s_t)=
\\
\begin{split}
\max_{(u_1, \ldots  ,u_l)\in A(i_t,l,s_t) }\bigg\{& (s_t-u_1-\cdots -u_l)D(i_1,\ldots ,i_{t-1},l,s_1,
\ldots ,s_{t-1})\\
&+u_1D(i_1, \ldots , i_{t-1},l-1,s_1, \ldots ,s_{t-1})+\cdots
\\
&+u_{l-1}D(i_1, \ldots ,i_{t-1},1,s_1, \ldots ,
s_{t-1})+u_ls_1\cdots s_{t-1} \bigg\}
\end{split}
\end{multline*}
where 
\begin{eqnarray}
A(i_t,l,s_t)&=&\{(u_1, \ldots ,u_l) \in {\mathbf{N}}_0^l \mid
u_1+2u_2+\cdots + l u_l \leq i_t\}\nonumber
\end{eqnarray}
follows from~(\ref{eqeqeq}). By the above assumptions this implies that
\begin{multline*}
D(i_1, \ldots ,i_t,l,s_1, \ldots ,s_t) \leq
\\
\shoveleft
\max_{(u_1, \ldots  ,u_l)\in B(i_t,l,s_t) } \bigg\{ s_t\big(
s_1 \cdots s_{t-1}-(s_1-\frac{i_1}{l})\cdots(s_{t-1}-\frac{i_{t-1}}{l})\big)\\
\begin{split}
&+u_1\big((s_1-\frac{i_1}{l})\cdots(s_{t-1}-\frac{i_{t-1}}{l})-(s_1-\frac{i_1}{l-1})\cdots(s_{t-1}-\frac{i_{t-1}}{l-1}) \big)\\
&+ \cdots
\\
&+u_{l-1}\big(
(s_1-\frac{i_1}{l})\cdots(s_{t-1}-\frac{i_{t-1}}{l})-(s_1-\frac{i_1}{1})\cdots(s_{t-1}-\frac{i_{t-1}}{1})
\big)\\
&+u_{l}\big((s_1-\frac{i_1}{l})\cdots(s_{t-1}-\frac{i_{t-1}}{l})\big) \bigg\}
\end{split}
\end{multline*}
where $$B(i_{t},l,s_t)=\{(u_1, \ldots , u_l) \in {\mathbf{Q}}^l \mid 0
\leq u_1, \ldots , u_l {\mbox{ \ and \ }} u_1+2u_2+\cdots +lu_l \leq i_t\}.$$
As $t<m $ condition  (A.2) applies and tells us that the maximal value
is attained for  $u_1=\cdots =u_{l-1}=0$ and
$u_l=\frac{i_t}{l}$. This concludes the induction proof of (\ref{eqminus1}).\\
To show (\ref{eqendnuenstjerne}) we apply similar arguments
to the case
$t=m$
but use condition (A.3) rather than condition (A.2). 
The proof of the last result in the proposition follows the proof of
the last part of Theorem~\ref{prorec}.
\end{proof}
\begin{Remark}
Experiments show (see~\cite{hmpage}) that the
bound in Theorem~\ref{prorec} is very often much better than  $\min \{
(i_1s_2\cdots s_m+\cdots + s_1 \cdots s_{m-1}i_m)/r,s_1\cdots s_m\}$,
however, they also reveal that in many cases one can get more information about the number of zeros by actually
applying the algorithm from Section~\ref{sec-cor-sz-gen-bet}.
\end{Remark}
\begin{Example}
This is a continuation of Example~\ref{exsmart} where we translated
Condition A into bounds on $i_1$ and $i_2$ in the case of
two variables. Applying in turn Proposition~\ref{prosmall} and
(C.3) in Proposition~\ref{protwovar} with $k=r-1$ we see that the two
bounds produce the very same values when $m=2$.
\end{Example}

{\section{Products of univariate linear terms}\label{seclin}}
\noindent
In this section we study the situation where $F(\vec{X})$ is
a product of univariate linear terms. First we note that
equivalently to Defintion~\ref{defmult} one can define the
multiplicity of a polynomial as
follows.
\begin{Definition}
Let $F(\vec{X}) \in {\mathbf{F}}[\vec{X}]\backslash \{0\}$ and
$\vec{a}=(a_1, \ldots , a_m) \in {\mathbf{F}}^m$. 
Consider the ideal 
\begin{eqnarray}
J_t=\langle (X_1-a_1)^{p_1}\cdots(X_m-a_m)^{p_m} \mid p_1+\cdots
+p_m=t \rangle \subseteq {\mathbf{F}}[X_1, \ldots , X_m]. \nonumber
\end{eqnarray}
We have
${\mbox{mult}}(F,\vec{a})=r$ if $F\in J_r \backslash
J_{r+1}$. If $F=0$ we have ${\mbox{mult}}(F,\vec{a})=\infty$.
\end{Definition}
The above definition makes it particularly simple to calculate the number of
zeros of multiplicity at least $r$ when $F$ is a product of univariate
linear terms. In the following write 
$$S_j=\{\alpha_1^{(j)}, \ldots , \alpha_{s_j}^{(j)} \}$$
for $j=1, \ldots , m$.
\begin{Proposition}\label{pronuogsaa}
Consider 
$$F(\vec{X})=\prod_{u=1}^{m} \prod_{v=1}^{s_u}(X_u-\alpha_v^{(u)})^{r_{v}^{(u)}}.$$
The multiplicity of $(\alpha_{j_1}^{(1)}, \ldots ,
\alpha_{j_m}^{(m)})$ in $F(\vec{X})$ equals
\begin{equation}
r_{j_1}^{(1)}+\cdots +r_{j_m}^{(m)}. \label{eqHjubi}
\end{equation}
\end{Proposition}
\begin{proof}
Without loss of generality assume
$j_1=\cdots =j_m=1$. Clearly, the multiplicity is greater than or equal to
$r=r_1^{(1)}+\cdots +r_1^{(m)}$. Using Gr\"{o}bner basis theory we now
show that it is not larger. We substitute
${\mathcal{X}}_i=X_i-\alpha_1^{(i)}$ for $i=1, \ldots ,m$ and observe
that by Buchberger's S-pair criteria 
$${\mathcal{B}}=\{{\mathcal{X}}_1^{r_1} \cdots {\mathcal{X}}_m^{r_m}
\mid r_1+\cdots +r_m=r+1\}$$
is a Gr\"{o}bner basis (with respect to any fixed monomial
ordering). 
The support of $F({\mathcal{X}}_1, \ldots , {\mathcal{X}}_m)$ contains
a monomial of the form ${\mathcal{X}}_1^{i_1} \cdots
{\mathcal{X}}_m^{i_m}$ with $i_1+\cdots +i_m=r$. 
 Therefore the remainder of $F({\mathcal{X}}_1, \ldots ,
{\mathcal{X}}_m)$ modulo ${\mathcal{B}}$ is non-zero. 
It is well known that if a
polynomial is reduced modulo a Gr\"{o}bner basis then the remainder is
zero if and only if it belongs to the ideal generated by the elements
in the basis.
\end{proof}

We now show that Theorem~\ref{prop-sz-gen} is tight. It follows of
course that so is Theorem~\ref{SZ-lemma} (a fact that has not been
stated in the literature). 
\begin{Proposition}
Let $S_1, \ldots , S_m \subseteq {\mathbf{F}}$ be finite sets. If
$F(\vec{X})\in {\mathbf{F}}[\vec{X}]$ is a product
of univariate linear factors  then the number of
zeros of $F$ counted with multiplicity reach the generalized Schwartz-Zippel bound
(Theorem~\ref{prop-sz-gen}).
\end{Proposition}
\begin{proof}
Consider the polynomial
$$F(\vec{X})=\prod_{u=1}^m \prod_{v=1}^{s_u}\big(X_u-\alpha_v^{(u)}\big)^{r_v^{(u)}}.$$
Write $i_u=\sum_{v=1}^{s_u}r_v^{(u)}$, $u=1, \ldots , m$. We have
\begin{align*}
\sum_{\vec{a} \in S_1 \times \cdots \times S_m}
{\mbox{mult}}(F,\vec{a})
&=\sum_{t=1}^{s_1}(s_2 \cdots s_m  )r_t^{(1)}+\cdots +
\sum_{t=1}^{s_m}(s_1 \cdots s_{m-1}  )r_t^{(m)}
\\
&=i_1s_2  \cdots s_{m}  +\cdots +s_1 
\cdots s_{m-1} i_m.
\qedhere
\end{align*}
\end{proof}

\begin{Example}\label{exbig}
Let $(i_1, \ldots , i_m, r, s_1,
\ldots ,s_m)$ be such that $\lfloor
i_1/s_1\rfloor+\cdots +\lfloor i_m/s_m\rfloor \geq r$.
 As mentioned in Remark~\ref{rembig} there exist
polynomials with the leading monomial being $X_1^{i_1} \cdots X_m^{i_m}$ such that all points in $S_1 \times \cdots \times S_m$ are
zeros of multiplicity at least $r$. To see this define $r_1=\lfloor i_1/s_1 \rfloor, \ldots ,
r_m=\lfloor i_m/s_m \rfloor$. Multiplying
$$\prod_{u=1}^m
\prod_{v=1}^{s_u}\big(X_u-\alpha_v^{(u)}\big)^{r_u}$$
by an appropriate monomial we get a polynomial having the prescribed
leading monomial (with respect to any monomial ordering). Clearly, all points in the ensemble are zeros of multiplicity at least $r$. 
\end{Example}
\begin{Definition}
Given $(i_1, \ldots , i_m,r, s_1, \ldots , s_m)$ consider the set of
polynomials that are products of univariate linear terms and have
$X_1^{i_1} \cdots X_m^{i_m}$ as leading monomial. By $H(i_1, \ldots
,i_m,r,s_1, \ldots , s_m)$ we denote the maximal number of zeros of
multiplicity at least $r$ that a polynomial from the above set can
have over $S_1 \times \cdots \times S_m$.
\end{Definition}
Based on Proposition~\ref{pronuogsaa} it is straightforward to
describe an iterative algorithm that finds $H(i_1, \ldots
, i_m,r,s_1, \ldots ,s_m)$. For the convenience of the reader we
include such an algorithm in Appendix~\ref{applin}.\\

In the previous sections we considered the general set of polynomials
$F$ with ${\mbox{lm}}_{\prec}(F)=X_1^{i_1} \cdots X_m^{i_m}$. We
gave upper bounds on the maximal attainable number of zeros of
multiplicity $r$ or more. It is clear that $H(i_1, \ldots , i_m,r,s_1,
\ldots ,s_m)$ serves as a lower bound on the maximal attainable number
of zeros of multiplicity $r$ or more. 
In particular 
$H(i_1, \ldots , i_m,r,s_1, \ldots , s_m) \leq D(i_1, \ldots ,
i_m,r,s_1, \ldots , s_m)$ holds. Experiments show (see~\cite{hmpage})
that the two functions are sometimes quite close. In the next
section we present various conditions under which the two functions
attain the same value. Clearly, when this happens we know what is the maximal number of
zeros of multiplicity at least $r$ that any polynomial with leading monomial
$X_1^{i_1} \cdots X_m^{i_m}$ can have over $S_1 \times \cdots \times S_m$.
\begin{Example}\label{exny4}
This is a continuation of Example~\ref{exny1} where we studied the
upper bound 
$D(i_1,i_2,3,5,5)$ for relevant choices of $i_1, i_2$. Using the
algorithm in Appendix~\ref{applin} we calculated the corresponding
values of the lower bound $H(i_1,i_2,3,5,5)$. We list in Table~\ref{tabny6} the
difference $D(i_1,i_2,3,5,5)-H(i_1,i_2,3,5,5)$. We see that for many
choices of $i_1,i_2$ the upper bound equals the lower bound.
\begin{table}
\centering
\caption{Difference between upper and lower bound in Example~\ref{exny4}}\label{tabny6}
\newcommand{\SP}{~\hspace*{0.851em}}
\begin{tabular}{@{}c@{}cc@{\SP}c@{\SP}c@{\SP}c@{\SP}c@{\SP}c@{\SP}c@{\SP}c@{\SP}c@{\SP}c@{~~~}c@{~~}c@{~~}c@{~~}c@{~~}c@{}}
    \toprule
    &&\multicolumn{15}{c}{$i_1$}\\
    &&0&1&2&3&4&5&6&7&8&9&10&11&12&13&14\\
    \addlinespace
    \multirow{15}{*}{$i_2$}
    &\multicolumn{1}{c}{0} &0&0&0&0&0&0&0&0&0&0&0&0&0&0&0\\
    &\multicolumn{1}{c}{1} &0&0&0&0&1&0&1&1&1&1&2&1&1&1&0\\
    &\multicolumn{1}{c}{2} &0&0&0&2&2&2&3&2&2&2&3&2&2&1&0\\
    &\multicolumn{1}{c}{3} &0&0&0&0&0&1&1&1&3&1&4&3&2&2&0\\
    &\multicolumn{1}{c}{4} &0&0&0&0&2&3&3&3&2&2&3&2&2&1&0\\
    &\multicolumn{1}{c}{5} &0&0&0&2&2&3&2&2&0&0\\
    &\multicolumn{1}{c}{6} &0&0&0&0&1&2&3&2&1&0\\
    &\multicolumn{1}{c}{7} &0&0&0&0&2&3&3&3&1&0\\
    &\multicolumn{1}{c}{8} &0&0&0&2&1&1&2&1&2&0\\
    &\multicolumn{1}{c}{9} &0&0&0&0&1&2&2&1&1&0\\
    &\multicolumn{1}{c}{10}&0&0&0&0&0\\
    &\multicolumn{1}{c}{11}&0&0&0&1&0\\
    &\multicolumn{1}{c}{12}&0&0&0&0&0\\
    &\multicolumn{1}{c}{13}&0&0&0&0&0\\
    &\multicolumn{1}{c}{14}&0&0&0&0&0\\
    \bottomrule
\end{tabular}
\end{table}
\end{Example}
\begin{Example}\label{exny5}
This is a continuation of Example~\ref{exny2} where we studied
$D(i_1,i_2,i_3=3,i_4=5,3,6,6,6,6,)$ for a collection of values
$i_1,i_2$. In Table~\ref{tabny7} we list the difference between these
upper bounds and the lower bounds $H(i_1,i_2,i_3=3,i_4=5,3,6,6,6,6)$.

\begin{table}
\centering
\caption{Difference between upper and lower bound in Example~\ref{exny5}}
\newcommand{\SP}{~\hspace*{0.851em}}
\begin{tabular}{@{}c@{~~}c@{\SP}c@{\SP}c@{\SP}c@{\SP}c@{\SP}c@{\SP}c@{\SP}c@{\SP}c@{}}
    \toprule
    & &\multicolumn{8}{c}{$i_1$} \\
    & &   0   &   1   &   2   &    3   &    4   &    5   &    6   &    7   \\
    \addlinespace
    \multirow{8}{*}{$i_2$}
    &0& 72& 60& 48& 96& 90&114&216&186\\
    &1& 60& 50& 60& 80&109&143&248&217\\
    &2& 48& 60& 53&104&133&179&265&190\\
    &3& 96& 80& 70&100&120&143&213&150\\
    &4& 90& 88&106&111&112&114&168&120\\
    &5&114&129&155&111& 82& 81&117& 84\\
    &6&216&196&201&112& 52& 54& 72& 54\\
    &7&186&181&146& 81& 40& 42& 57& 42\\
    \bottomrule
\end{tabular}
\label{tabny7}
\end{table}
\end{Example}
\begin{Example}
Let $s_1=\cdots =s_m=q$. Our experiments listed in~\cite{hmpage} show that  $D(i_1, \ldots ,
i_m,r,q,\ldots ,q)$ is often close to $H(i_1, \ldots , i_m,r,q,\ldots
,q)$. In Table~\ref{tabheart} we list the mean value of
\begin{equation}
\frac{D(i_1, \ldots , i_m,r,q,\ldots , q)-H(i_1, \ldots , i_m,r,q,\ldots , q)}{\frac{1}{2}\big(D(i_1, \ldots , i_m,r,q,\ldots , q)+H(i_1, \ldots , i_m,r,q,\ldots , q)\big)}.\label{eqheart}
\end{equation}
The average is taken over the set of exponents with 
$\lfloor i_1/q\rfloor+\cdots +\lfloor i_m/q\rfloor 
< r$  and $D(i_1,\ldots,i_m,r,q,\ldots, q)\neq 0$.
\begin{table}[!h]
\centering
\caption{Mean value of (\ref{eqheart}); rounded up}
\newcommand{\SP}{~~}
\begin{tabular}{@{}c@{}c@{~~~~}r@{.}l@{\SP}r@{.}l@{\SP}r@{.}l@{\SP}r@{.}l@{\SP}r@{.}l@{\SP}r@{.}l@{\SP}r@{.}l@{\SP}r@{.}l@{\SP}r@{.}l@{\SP}r@{.}l@{}}
\toprule
$m$&& \multicolumn{8}{c}{2} & \multicolumn{8}{c}{3} & \multicolumn{4}{c}{4} \\
\cmidrule(r){3-10} \cmidrule(r){11-18} \cmidrule{19-22}
$r$&&\multicolumn{2}{l}{2}&\multicolumn{2}{l}{3}&\multicolumn{2}{l}{4}&\multicolumn{2}{l}{5}&\multicolumn{2}{l}{2}&\multicolumn{2}{l}{3}&\multicolumn{2}{l}{4}&\multicolumn{2}{l}{5}&\multicolumn{2}{l}{2}&\multicolumn{2}{l}{3}\\
\addlinespace
\multirow{7}{*}{$q$}
&\multicolumn{1}{c}{2}&0&044&0&066&0&08 &0&088&0&048&0&085&0&106&0&120&0&041&0&081\\
&\multicolumn{1}{c}{3}&0&039&0&048&0&068&0&075&0&046&0&067&0&092&0&104&0&038&0&064\\
&\multicolumn{1}{c}{4}&0&044&0&057&0&054&0&067&0&049&0&075&0&083&0&098&0&037&0&068\\
&\multicolumn{1}{c}{5}&0&042&0&057&0&061&0&060&0&045&0&073&0&086&0&092&0&034&0&064\\
&\multicolumn{1}{c}{6}&0&041&0&057&0&065&0&066&0&043&0&072&0&087&0&095&0&031&0&061\\
&\multicolumn{1}{c}{7}&0&040&0&054&0&063&0&066&0&042&0&069&0&085&0&094&0&030&0&058\\
&\multicolumn{1}{c}{8}&0&038&0&053&0&062&0&067&0&040&0&067&0&083&\multicolumn{2}{l}{?}&0&029&0&056\\
\bottomrule
\end{tabular}
\label{tabheart}
\end{table}
\end{Example}
{\section{Some conditions for  $H = D$ to hold}\label{secequal}}
As the polynomial ring in one variable is a unique factorization
domain we get $H(i_1,r,s_1)=D(i_1,r,s_m)$ for all choices of $i_1, r,
s_1$. Experiments suggest (see~\cite{hmpage}) that for two
variables we have a similar equality for certain systematic choices of
$i_1, i_2$. For other choices of $i_1, i_2$ the picture is more
blurred. The results of our experiments further suggest that it might not be
an easy task 
to say much about which values of $(i_1, \ldots . i_m, r, s_1, \ldots
, s_m)$ causes equality when $m \geq 3$. We summarize our
findings below.
\begin{Proposition}
For $\frac{r}{r+1}s_1 \leq i_1 < s_1$, $(r-1)s_2 \leq i_2 <rs_2$ we have
\begin{eqnarray}
H(i_1, i_2,r,s_1,s_2)=D(i_1,i_2,r,s_1,s_2) =rs_2i_1+i_2s_1-i_1i_2-(r-1)s_1s_2.\nonumber 
\end{eqnarray}
\end{Proposition} 
\begin{proof}
The value of $D$ is upper bounded by (C.2)
in Proposition~\ref{protwovar}. The value of $H$ is lower bounded by
studying the zeros of
\begin{multline*}
(X_2-\alpha_1^{(2)})^r\cdots
(X_2-\alpha_w^{(2)})^r(X_2-\alpha_{w+1}^{(2)})^{r-1} \cdots
\\
(X_2-\alpha_{s_2}^{(2)})^{r-1}(X_1-\alpha_1^{(1)})\cdots
(X_1-\alpha_{s_1}^{(1)})
\end{multline*}
where $w=i_2-(r-1)s_2$.
\end{proof}

We leave the proofs of the following two results for the reader.
\begin{Proposition}
Assume $r \leq s_1$. If $0 \leq i_1 < r$ and $0 \leq i_2 < rs_2$ holds
then
$$H(i_1,i_2,r,s_1,s_2)=D(i_1,i_2,r,s_1,s_2)=\lfloor i_2/r\rfloor
s_2+\delta$$
where $\delta=i_1-(r-w)+1$ if $r-w \leq i_1$ and $\delta=0$ otherwise.
\end{Proposition}
\begin{Proposition}
If $H(i_1, \ldots , i_m,r,s_1, \ldots ,s_m)=D(i_1, \ldots , i_m,r,s_1,
\ldots ,s_m)$ then
\begin{align*}
&H(i_1, \ldots , i_t,0,i_{t+1}, \ldots , i_m,r,s_1, \ldots ,
s_t,s^{\prime},s_{t+1}, \ldots  ,s_m)\\
&=D(i_1, \ldots , i_t,0,i_{t+1}, \ldots , i_m,r,s_1, \ldots ,
s_t,s^{\prime},s_{t+1}, \ldots  ,s_m)\\
&=s^{\prime}D(i_1, \ldots , i_m,r,s_1, \ldots ,s_m).
\end{align*}
\end{Proposition}
\begin{Proposition}
Assume $(i_1, \ldots  ,i_m,r,s_1, \ldots , s_m)$ satisfies Condition A
(Definition~\ref{defconda}) and that $r$ divides $i_1, \ldots ,
i_m$. Then
\begin{align}
H(i_1, \ldots , i_m,r,s_1, \ldots ,s_m)
&=D(i_1, \ldots , i_m,r,s_1, \ldots ,s_m) \nonumber \\
&=s_1\cdots s_m-(s_1-\frac{i_1}{r}) \cdots (s_m-\frac{i_m}{r}). \nonumber
\end{align}
\end{Proposition}
\begin{proof}
Consider
$$\prod_{j=1}^m\prod_{v=1}^{i_j/r} (X_j-\alpha_v^{(j)})^r$$
and apply Proposition~\ref{prosmall}.
\end{proof}

{\section{Acknowledgments}\label{hallihallo}}
The authors wish to thank T.\ Mora, P.\ Beelen, D.\ Ruano
and T.\ H{\o}holdt for pleasant discussions. Thanks to L.\ Grubbe Nielsen
for linguistical assistance.
\appendix
{\section{An algorithm to calculate $H$}\label{applin}}
We here give the details of the algorithm mentioned in
Section~\ref{seclin}.
\begin{Definition}
Consider vectors $$\vec{v}^{(1)}=(v_1^{(1)}, \ldots  ,v_r^{(1)}),
\ldots , \vec{v}^{(m)}=(v_1^{(m)}, \ldots  ,v_r^{(m)})\in
{\mathbf{N}}_0^r.$$
Let $s_1, \ldots ,s_m\in {\mathbf{N}}$. Define for $k=1,\ldots ,r$
$$\tilde{H}(\vec{v}^{(1)},k,s_1)=v_{k}^{(1)}+\cdots + v_r^{(1)}$$
and for $k \leq r$ and $m\geq 2$
\begin{multline*}
\tilde{H}(\vec{v}^{(1)}, \ldots , \vec{v}^{(m)},k,s_1, \ldots ,
s_m)=\\
\begin{split}
&\big[ s_m - (v_1^{(m)}+\cdots +v_r^{(m)})\big] \tilde{H}(\vec{v}^{(1)},
\ldots , \vec{v}^{(m-1)},k, s_1, \ldots , s_{m-1})\\
&+v_1^{(m)}\tilde{H}(\vec{v}^{(1)},
\ldots , \vec{v}^{(m-1)},k-1, s_1, \ldots , s_{m-1})\\
&+\cdots +  v_{k-1}^{(m)}\tilde{H}(\vec{v}^{(1)},
\ldots , \vec{v}^{(m-1)},1, s_1, \ldots , s_{m-1})\\
&+\tilde{H}(\vec{v}^{(m)},k,s_m)s_1 \cdots s_{m-1}.
\end{split} 
\end{multline*} 
\end{Definition}
\begin{Proposition}
\begin{multline*}
H(i_1,\ldots ,i_m,r,s_1, \ldots , s_m)=
\\
\max \bigg\{ \tilde{H}(\vec{v}^{(1)},
\ldots , \vec{v}^{(m)},r,s_1, \ldots , s_m) \bigg|
v_1^{(t)}+\cdots +v_r^{(t)} \leq s_t,\\
v_1^{(t)}+2v_2^{(t)}+\cdots +rv_r^{(t)}=i_t, {\mbox{ \ for \ }} t=1,
\ldots , m \bigg\}
\end{multline*}
\end{Proposition}

{\section{Comparison of Theorem~\ref{SZ-lemma} to a bound by
  Pellikaan and Wu}\label{seccomp}}
\noindent
As mentioned in the introduction for $S_1=\cdots =S_m={\mathbf{F}}_q$ there is an
alternative to Dvir et al.'s method from~\cite{dvir}, namely the
method by Pellikaan and Wu in~\cite{pw_ieee} and
\cite{pw_expanded}. We conclude the paper by showing that this other
approach is never better than Corollary~\ref{corsz}. Thus the
results in the present paper are the best known results.\\

In~\cite{pw_ieee} Pellikaan and Wu presented two algorithms
for decoding generalized Reed-Muller codes. The first algorithm is
based on 
the fact that generalized Reed-Muller codes can be viewed as subfield
subcodes of Reed-Solomon codes whereas the second algorithm is a
straightforward generalization of the Guruswami-Sudan decoding
algorithm in~\cite{GS}. The analysis of the second algorithm
in~\cite{pw_ieee} relies on a generalization of the footprint bound
from Gr\"{o}bner basis theory. As the first algorithm outperforms the
second, the details of the analysis of the second are not included
in the journal paper~\cite{pw_ieee} but can be found
in~\cite{pw_expanded}. To state the generalization of the footprint
bound we will need the following two lemmas corresponding
to~\cite[Lemma 2.4]{pw_expanded} respectively \cite[Lemma
2.5]{pw_expanded}.
\begin{Lemma}\label{lempw1}
Given a polynomial $F(\vec{X}) \in {\mathbf{F}}_q[\vec{X}]$ consider
the ideal $$I(q,r,m,F)= \langle F \rangle +\langle (X_1^q-X_1)^{e_1}
\cdots (X_m^q-X_m)^{e_m} \mid e_1+ \cdots e_m=r \rangle.$$ If $t$ is
the number of points in ${\mathbf{F}}_q^m$ where $F$ has at least
multiplicity $r$, then
$$\dim_{{\mathbf{F}}_q} {\mathbf{F}}_q[X_1, \ldots , X_m]/I(q,r,m,f)
\geq {m+r-1 \choose r-1} t.$$
\end{Lemma}
\begin{Lemma}\label{lempw2}
Let $d$ be the total degree of $F(\vec{X}) \in {\mathbf{F}}_q[\vec{X}]$ and define $w=\lfloor d/ q \rfloor$. If $d < qr$ then an upper
bound for the dimension of 
$${\mathbf{F}}_q[\vec{X}]/I(q,r,m,f)$$ 
is given by
$${m+r-1 \choose m}q^m+(d-qw){m+r-w-2 \choose m-1}q^{m-1}-{m+r-w-1
  \choose m}q^m.$$
\end{Lemma}
Combining the two lemmas above we get the following result which is
used in~\cite{pw_expanded} without being stated explicitly. 
\begin{Proposition}\label{pwbound}
Let the notation be as in the above lemmas and assume $d < qr$. The number of points in ${\mathbf{F}}_q^m$ where $F$ has at least
multiplicity $r$ is at most equal to
\begin{eqnarray}
\Gamma_1(q,r,m,d)=\frac{{m+r-1 \choose m}q^m+(d-qw){m+r-w-2 \choose
    m-1}q^{m-1}-{m+r-w-1 \choose m}q^m}{{m+r-1 \choose r-1 }}.\nonumber
\end{eqnarray}
\end{Proposition}
Augot and Stepanov~\cite{augot} gave another interpretation of
Pellikaan and Wu's second decoding algorithm in~\cite{pw_ieee} by
using Theorem~\ref{SZ-lemma} instead of
Proposition~\ref{pwbound}. Doing this they were able to correct much more errors which indicates
that the generalized Schwartz-Zippel bound is stronger than
Proposition~\ref{pwbound}. We here provide a direct proof of this
fact.
\begin{Proposition}
Let $\Gamma_2(q,r,m,d)=d q^{m-1}/r$, then
\begin{eqnarray}
\Gamma_1(q,r,m,d) \geq \Gamma_2(q,r,m,d) \nonumber
\end{eqnarray}
holds for all $d \in [ 0,rq-1]$.
\end{Proposition}
\begin{proof}
We consider $\Gamma_1$ and $\Gamma_2$ as
functions in $d$ on the interval $[0,rq]$. Our first observation is
that $\Gamma_1$ is a continuously piecewise linear function,
each piece corresponding to a particular value of $w$. The
corresponding $w$ slopes constitute a decreasing sequence. Combining this
observation with the fact that $\Gamma_2$  is linear in $d$ and with
the fact that
$$\Gamma_1(q,r,m,0)=\Gamma_2(q,r,m,0) {\mbox{ \ and  \ }}
\Gamma_1(q,r,m,rq)=\Gamma_2(q,r,m,rq)$$
proves the proposition.
\end{proof}

\end{document}